\documentclass{article}

\usepackage{styles/VABcommands}
\usepackage{styles/VABenvironments}
\usepackage[backend=bibtex, backref=true, useprefix, maxcitenames=2]{biblatex}

\usepackage{thmtools}
\usepackage{thm-restate}

\usepackage{enumitem}
\setenumerate[1]{label=\upshape{\arabic*)}} %

\title{Littlewood--Paley--Rubio de Francia Inequality for the Two-parameter Walsh System}
\author{
Viacheslav Borovitskiy
\thanks{
This research was supported by the Ministry of Science and Higher Education of the Russian Federation (agreement No. 075-15-2019-1620), and by the Foundation for the Advancement of Theoretical Physics and Mathematics ``BASIS''.
\vspace{0.5\baselineskip}\newline
\emph{Keywords}: Littlewood-Paley inequality, Rubio de Francia inequality, Walsh system, Gundy's theorem, martingale, Hardy space, two-parameter, multi-parameter singular integral operator. 
}
\\
{\small St. Petersburg Department of Steklov Mathematical Institute}
\\
{\small St. Petersburg State University}
}
\date{\vspace{-5ex}}
\bibliography{references}

\begin{document}

\maketitle

\begin{abstract}
    A version of Littlewood--Paley--Rubio de Francia inequality for the two-parameter Walsh system is proved: for any family of disjoint rectangles $I_k = I_k^1 \x I_k^2$ in ${\Z_+ \x \Z_+}$  and a family of functions $f_k$ with Walsh spectrum inside $I_k$ the following is true
    $$
    \norm{\sum\limits_k f_k}_{L^p} \leq C_p
    \norm{\del{\sum\limits_k \abs{f_k}^2}^{1/2}}_{L^p}
    , \qquad 1 < p \leq 2,
    $$
    where $C_p$ does not depend on the choice of rectangles $\cbr{I_k}$ or functions~$\cbr{f_k}$.
    The arguments are based on the atomic theory of two-parameter martingale Hardy spaces.
    In the course of the proof, a two-parameter version of Gundy's theorem on the boundedness of operators taking martingales to measurable functions is formulated, which might be of independent interest.
\end{abstract}

\section{Introduction}

Consider a countable index set $\mathbf{Z}$ and an orthonormal basis $\cbr{\phi_n}_{n \in \mathbf{Z}}$ in the space $L^2$.
Define operators $M_I$ for $I \subseteq \mathbf{Z}$ by the expression $M_I f = \sum_{n \in I} \innerprod{f}{\phi_n} \phi_n$.
Whenever $M_{I} f = f$, we say that the spectrum of $f$ lies in $I$ and write $\spec{f} \subseteq I$.

Consider also a partition $\cbr{I_k}_{k \in \N}$ of the index set $\mathbf{Z}$ and a family of functions~$f_k \in L^2$ such that $\spec{f_k} \subseteq I_k$.
Then the following equality holds
\[ \label{eqn:parseval_reformulation}
\norm{\sum\limits_{k = 1}^\infty f_k}_{L^2} = \norm{\del{\sum\limits_{k = 1}^\infty \abs{f_k}^2}^{1/2}}_{L^2}
.
\]
This follows directly from Parseval's identity, generalizing, in a sense, this classical result: if the $I_k$ are singletons, we recover precisely Parseval's identity.

Of course, if we replace both $L^2$ norms in equation \eqref{eqn:parseval_reformulation} by $L^p$ norms with some $p \not= 2$, the identity will not be valid.
In this case it is interesting to study a weaker kind of relationships between the left hand side and the right hand side of \eqref{eqn:parseval_reformulation}.
For instance, for some bases $\cbr{\phi_n}$ and specific partitions $\mathbf{Z} = \cup_{k \in \N} I_k$, this or that one-sided inequality with a multiplicative constant might be true.

The most famous assertion of this kind is the Littlewood--Paley inequality
\[ \label{eqn:littlewood-paley}
c_p \norm{\del{\sum\limits_{k = 1}^\infty \abs{f_k}^2}^{1/2}}_{L^p}
\leq
\norm{\sum\limits_{k = 1}^\infty f_k}_{L^p}
\leq
C_p \norm{\del{\sum\limits_{k = 1}^\infty \abs{f_k}^2}^{1/2}}_{L^p}
,~~1 < p < \infty,
\]
where $\phi_n(t) = e^{2 \pi i n t}$, $n \in \Z$ is a standard trigonometric system over the interval~$[0,1]$, $L^2 = L^2([0,1])$, and $I_k$ is a partition of the set $\Z$ of integers in a Hadamard lacunary sequence of intervals.
\footnote{The very same year when Littlewood and Paley introduced the pair of inequalities \eqref{eqn:littlewood-paley} for the trigonometric system (see \cite{littlewood1931}), the paper \cite{paley1931} of Paley appeared, proving the same pair of inequalities for the Walsh system that we will study in this paper.}

The corresponding statement for trigonometric system and partitions of $\Z$ into arbitrary intervals was established by Rubio de Francia \cite{rubiodefrancia1985} in 1985.
He showed that in this case the following pair of inequalities holds
\<
\label{eqn:rdf12}
\norm{\sum\limits_{k = 1}^\infty f_k}_{L^p}
&\leq
C_p
\norm{\del{\sum\limits_{k = 1}^\infty \abs{f_k}^2}^{1/2}}_{L^p}
, \qquad &1 < p \leq 2,
\\
\label{eqn:rdf2inf}
c_p
\norm{\del{\sum\limits_{k = 1}^\infty \abs{f_k}^2}^{1/2}}_{L^p}
&\leq
\norm{\sum\limits_{k = 1}^\infty f_k}_{L^p}
, \qquad &p \geq 2.
\>
These are called Littlewood--Paley--Rubio de Francia inequalities or simply Rubio de Francia inequalities.
Establishing these sparked a whole new line of research, yielding a number of extensions of this result published to date.

The majority of the extensions study the case of the trigonometric system.
Specifically, in the papers \cite{bourgain1985, kislyakov2005} inequality \eqref{eqn:rdf12} was generalized to arbitrary exponents $0 < p \leq 2$.
In the papers \cite{journe1985, soria1987, osipov2010a, osipov2010b} generalization for the $D$-parameter trigonometric system $\phi_n$, $n \in \Z^D$, and partitions of $\Z^D$ into arbitrary products of intervals was formulated, with inequality \eqref{eqn:rdf12} similarly extended to arbitrary exponents~$0 < p \leq 2$.
Rubio de Francia himself in the original paper \cite{rubiodefrancia1985} as well as other authors in \cite{kislyakov2008,borovitskiy2020} considered some weighted generalizations.
In the papers~\cite{osipov2012,malinnikova2019} some versions of these inequalities for the Morrey--Companato and Tribel--Lizorkin spaces were proved.
There is a review by Lacey \cite{lacey2003} that considers some of the mentioned, as well as some other extensions of Rubio de Francia inequalities.

Recently, Osipov \cite{osipov2017} proved a version of inequality \eqref{eqn:rdf12} where $\cbr{\phi_n}$, $n \in \Z_+$, is the Walsh system and the $I_k$ partition the positive integers $\Z_+$ into arbitrary pairwise nonintersecting intervals.
In this paper we take this line of research further by proving \eqref{eqn:rdf12} for the two-parameter Walsh system $\phi_n$, $n \in \Z_+ \x \Z_+$ and partitions of ${\Z_+ \x \Z_+}$ into arbitrary pairwise nonintersecting rectangles.
Formally, we prove the following statement.

\begin{restatable}{theorem}{walshrdftwodim}
\label{th:main_theorem}
Consider a family of pairwise nonintersecting rectangles $I_k= I_k^1 \x I_k^2$ inside ${\Z_+ \x \Z_+}$ and a family of functions $f_k$ with Walsh spectrum inside $I_k$, meaning that
\[
f_k(x_1, x_2) = \sum\limits_{(n_1, n_2) \in I_k} (f_k, w_{n_1} w_{n_2}) w_{n_1}(x_1) w_{n_2}(x_2),
\]
where $w_{n_i}$ are the standard Walsh functions in the Paley ordering.

\noindent If $1 < p \leq 2$, then
\[
\norm{\sum\limits_k f_k}_{L^p}
\leq
C_p
\norm{\del{\sum\limits_k \abs{f_k}^2}^{1/2}}_{L^p}
,
\]
where $C_p$ does not depend on the choice of rectangles $\cbr{I_k}$ or functions $\cbr{f_k}$.
\end{restatable}

The proof is based upon a martingale version of the two-parameters singular integral theory of R. Fefferman and Journe~\cite{fefferman1986,journe1985}, formulated by Weisz~\cite{weisz1997}.

We use the theory of Weisz to prove Theorem~\ref{th:gundy2d} --- a two-parameter analog of Gundy's Theorem \cite{gundy1968} on the boundedness of operators taking martingales to measurable functions.
Theorem~\ref{th:gundy2d} helps proving the boundedness of operators that map two-parameter martingales to measurable functions in a rather general setting, and thus stands out as interesting on its own.
\footnote{To the best knowledge of the author, this assertion has not been explicitly formulated in the contemporary literature.}

By applying the combinatorial argument from Osipov's work on the one-parameter Walsh system \cite{osipov2017} independently for each variable, we essentially reduce the Rubio de Francia inequality for the two-parameter Walsh system to the question of boundedness for a certain operator, which in its turn we resolve by means of Theorem~\ref{th:gundy2d}.

\section{Preliminaries} \label{sec:preliminaries}

Here we present some preliminaries that will help us prove the main Theorem~\ref{th:main_theorem}.
First, we define two-parameter dyadic martingales and introduce the corresponding Hardy spaces.
Then we present some notions from the atomic theory of Hardy spaces that are useful for establishing boundedness of operators mapping martingales to measurable functions.
Finally, we recall the definition of the classical Walsh basis and define the two-parameter Walsh system.

Though we will need the theory of $l^2$-valued functions and martingales to prove the main theorem, to avoid cumbersome notation we study in this section only the scalar-valued case.
We do so because every definition, notion and assertion introduced here will be trivially transferable to the $l^2$-valued case.

\subsection{Two-parameter dyadic martingales} \label{sec:two_param_dyadic_martingales}

We define \emph{two-parameter dyadic filtration} to be the family $\cbr{\c{F}_{n_1, n_2}}_{n_1 \in \Z_+, n_2 \in \Z_+}$ of $\sigma$-algebras generated by the dyadic rectangles of size $2^{-n_1} \x 2^{-n_2}$, that is
\[
\c{F}_{n_1, n_2}
=
\sigma
\del{
    \cbr{
        \sbr{\frac{k_1}{2^{n_1}}, \frac{k_1 + 1}{2^{n_1}}}
        \times
        \sbr{\frac{k_2}{2^{n_2}}, \frac{k_2 + 1}{2^{n_2}}}
        :
        0 \leq k_i < 2^{n_i}
    }
}
,
\]
where $\sigma(\c{H})$ denotes the $\sigma$-algebra generated by the elements of the set $\c{H}$.
Define operator $\E_{n_1, n_2}$~to be the conditional expectation with respect to the $\sigma$-algebra~$\c{F}_{n_1, n_2}$.

Hereinafter we will often denote elements $(n_1, n_2) \in \Z_+^2$ by a single symbol~$n$. 
For $n, m \in \Z_+^2$ we write $n \leq m$ if and only if $n_1 \leq m_1$ and $n_2 \leq m_2$. With this, we introduce the following definition.

\begin{definition}
A family of integrable functions $u = \cbr{u_{n}}_{n \in \Z_+^2}$ is a \emph{two-parameter dyadic martingale} (from now on referred to as \emph{a martingale}) if the following conditions are fulfilled:
\1 for all $n \in \Z_+^2$ the function $u_n$ is $\c{F}_{n}$-measurable,
\2 we have $E_n u_m = u_n$ for all $n, m$ such that $n \leq m$.
\0
\end{definition}

We say that a martingale $u$ is in $L^p$ and write $u \in L^p$ for some $0 < p \leq \infty$ if $u_n \in L^p$ for all $n \in \Z_+^2$ and  $\norm{u}_{L^p} := \sup_{n \in \Z_+^2} \norm{u_n}_{L^p} < \infty$.
For two-parameter martingales, as in the classical one-parameter case, the following is true~\cite{weisz2006}: if~$u \in L^p$ for $1 < p < \infty$, then there exists a function $g \in L^p$ such that $u_n = \E_n g$ and
\[\lim\limits_{\min(n_1, n_2) \to \infty} \norm{u_n - g}_{L^p} = 0
,
\qquad
\qquad
\qquad
\norm{u}_{L^p} = \norm{g}_{L^p}
.
\]
Following the common practice, we will henceforth identify a martingale $u$ with the function $g$ and denote $g$ by the same symbol $u$.

Another important objects that we define are the \emph{martingale differences} $\Delta_n$:
\[
\Delta_{n_1, n_2} u
:=
u_{n_1, n_2} - u_{n_1 - 1, n_2}
- u_{n_1, n_2 - 1} + u_{n_1 - 1, n_2 - 1}
,
\]
where the formal symbols $u_{n_1, -1}$ and $u_{-1, n_2}$ are assumed to be equal to zero.

\subsection{Hardy spaces of two-parameter dyadic martingales}

We start with introducing a version of the Littlewood--Paley square function for two-parameter dyadic martingales.
\begin{definition}
    \emph{Littlewood--Paley square function} is denoted by $S$ and is given by
    \[
    S(u) := \del[3]{\sum\limits_{n \in \Z_+^2} \abs{\Delta_n u}^2}^{1/2}.
    \]
\end{definition}
The expression $\norm{S(u)}_{L^p}$ constitutes a norm that defines Hardy spaces.
\begin{definition}
    For $0 < p < \infty$ the \emph{martingale Hardy space} $H^p$ (from now on referred to as \emph{the Hardy space}) consists of martingales $u$ such that
    \[
    \norm{u}_{H^p} := \norm{S(u)}_{L^p} < \infty.
    \]
\end{definition}
It is known (cf. \cite{brossard1980,brossard1981,metraux1978}) that $\norm{S(u)}_{L^p} \sim \norm{u}_{L^p}$ for $1 < p < \infty$, meaning that for such exponents $p$ the spaces $L^p$ and $H^p$ coincide.
However for $p \leq 1$, the Hardy spaces constitute an independent and very useful entity.

We finish with formulating the following interpolation result for Hardy spaces.
\begin{theorem} \label{th:interpolation}
Consider a sublinear operator $V$ that is bounded between $H^{p_0}$ and $L^{p_0}$ and between $H^{p_1}$ and $L^{p_1}$.
Then $V$ is bounded between $H^p$ and $L^p$ for $p_0 < p < p_1$.
\end{theorem}
\begin{proof}
Cf. \cite[Theorem A]{weisz1997}.
\end{proof}

\subsection{Boundedness of operators in~$H^p$}

R. Fefferman's theorem \cite{fefferman1986} is an extremely important tool for establishing the boundedness of operators on two-parameter Hardy classes in trigonometric harmonic analysis.
It allows one to check rather simple quasi locality conditions, similar to those often used in one-parameter case.
As it turns out, the situation in the two-parameter martingale case is similar.
The corresponding claim is based, as in the trigonometric case, on the atomic decomposition of Hardy space.
Here we will formulate this claim.
We start with two definitions.

First, we define the martingale counterpart of R. Fefferman's rectangle atoms.
\begin{definition} \label{dfn:rectangle_atom}
    We call a function $a \in L^2$ a \emph{martingale  $H^p$ rectangle atom} (from now on referred to as a \emph{rectangle atom}) if the following conditions are satisfied
    \1 $\supp a \subseteq F$, where $F \subseteq [0, 1)^2$ is some dyadic rectangle,
    \2 $\norm{a}_{L^2} \leq \abs{F}^{1/2-1/p}$,
    \3 for all $x, y \in [0,1)$ we have $\int\limits_0^1 a(u, y) du = \int\limits_0^1 a(x, u) du = 0$.
    \0
\end{definition}
In accordance with the convention mentioned in Subsection \ref{sec:two_param_dyadic_martingales}, we view rectangle atom as a function or a martingale depending on the context.

Second, we introduce a class of operators for which the aforementioned quasi locality condition is satisfied.
\begin{definition}
    An operator $V$ mapping martingales to measurable functions is said to be \emph{$H^p$ quasi local}, if there exists $\delta > 0$ such that for all $r \in \N$, for all dyadic rectangles $R \subseteq [0,1)^2$, and for all $H^p$ rectangle atoms supported on $R$ we have
    \[
    \int\limits_{[0,1)^2 \setminus R_r} \abs{V a}^p \leq C_p 2^{- \delta r},
    \]
    where $R_r$ is a dyadic rectangle such that $R \subseteq R_r$ and $\abs{R_r} = 2^{2 r} \abs{R}$, and $C_p$ is a constant depending only on $p$.
\end{definition}

Finally, we formulate the claim.
\begin{theorem} \label{th:atomic_bounded}
Consider a sublinear operator $V$ that is $H^p$ quasi local for some exponent $0< p \leq 1$.
If $V$ is bounded between $L^2$ and $L^2$, then
\[
\norm{V u}_{L^p} \leq C_p \norm{u}_{H^p}
\t{~for all~}
u \in H^p
.
\]
\end{theorem}
\begin{proof}
Cf. \cite[Theorem 2]{weisz1997}.
\end{proof}

\subsection{Two-parameter Walsh system} \label{sec:Walsh2D}
We conclude the preliminaries with defining the two-parameter Walsh system.
We start with recalling the definition of the classical one-parameter Walsh system.
\begin{definition}
    \emph{The Walsh system} $\cbr{w_n}_{n \in \Z_+}$ is a family of piecewise constant functions of one real variable defined as follows.
    First, put $w_0 = 0$.
    Then, if $n > 0$ and $n = 2^{k_1} + \dots + 2^{k_s}$, $k_1 > k_2 > \dots > k_s \geq 0$, put
    \[
    w_n(x) := \prod_{i=1}^s r_{k_i + 1}(x),
    \t{~~where~~} r_k(x) = \sgn \sin 2^k \pi x
    .
    \]
\end{definition}
Here $\cbr{r_k}_{k \in \Z_+}$ is the Rademacher system.
Different orderings of Walsh functions are considered throughout the literature.
The ordering that corresponds to the definition above is called \emph{the Paley ordering}.
Hereinafter we consider precisely this ordering.

The \emph{two-parameter Walsh system} is defined by the expression
\[
w_{n_1, n_2} (x_1, x_2)
=
w_{n_1}(x_1) w_{n_2}(x_2),
\qquad
n \in Z_+^2.
\]
It is an orthonormal basis in $L^2 \del{[0,1]^2}$.
Moreover, for any function $f$ we have
\<
\del{{\E}_{k_1, k_2} f} (x_1, x_2)
&=
\sum\limits_{n_1 = 0}^{2^{k_1} - 1}
    \sum\limits_{n_2 = 0}^{2^{k_2} - 1}
        \innerprod{f}{w_n}
        \ 
        w_n(x_1, x_2),
\\
\del{\Delta_{k_1, k_2} f} (x_1, x_2)
&=
\sum\limits_{n \in \delta_{k_1, k_2}}
        \innerprod{f}{w_n}
        \ 
        w_n(x_1, x_2),
\>
where $\innerprod{\cdot}{\cdot}$ is the inner product in $L^2 \del{[0,1]^2}$ and
\[
\begin{aligned} \label{eqn:delta_k}
\delta_{k_1, k_2}
&=
[2^{k_1-1}, 2^{k_1}-1]
\x
[2^{k_2-1}, 2^{k_2}-1]
, \qquad &k_1, k_2 > 0,
\\
\delta_{0, k_2}
&=
\cbr{0}
\x
[2^{k_2-1}, 2^{k_2}-1]
,  &k_2 > 0,
\\
\delta_{k_1, 0}
&=
[2^{k_1-1}, 2^{k_1}-1]
\x
\cbr{0}
, &k_1 > 0,
\\
\delta_{0, 0}
&=
\cbr{(0,0)}
.
\end{aligned}
\]

For a pair $w_n(\cdot), w_m(\cdot)$, $n, m \in \Z_+$ of two one-parameter Walsh functions we have $w_n (x) w_m(x) = w_{n \dot{+} m}(x)$, where $\dot{+}$ is the bitwise exclusive disjunction (xor) operation acting upon the binary representations of numbers $n$ and $m$:
\[
\del[2]{{\sum}_{k=0}^\infty \alpha_k 2^k}
\dot{+}
\del[2]{{\sum}_{k=0}^\infty \beta_k 2^k}
:=
{\sum}_{k=0}^\infty (\alpha_k + \beta_k \bmod 2) 2^k
.
\]
If we define the corresponding operation $\dot{+}$ acting on a pair of $n = (n_1, n_2)$ and $m = (m_1, m_2)$ by putting
\[
n \dot{+} m = (n_1 \dot{+} m_1, n_2 \dot{+} m_2)
,
\]
then obviously
\[
w_n (x_1, x_2) w_m(x_1, x_2) = w_{n \dot{+} m}(x_1, x_2)
.
\]

\section{Two-parameter Gundy theorem}

Theorem \ref{th:atomic_bounded} from the previous section enables us to formulate a new version of Gundy's theorem that he introduced in the papers \cite{gundy1968, gundy1980}.
Note that our formulation will be closer to the version formulated much later in the Kislyakov's paper \cite{kislyakov1987}.
This theorem, due to the simplicity of its conditions, can be notably useful in proving the boundedness of operators taking martingales into measurabe functions.
It is thus an interesting result on its own and the key to proving Theorem \ref{th:main_theorem}.

We with a definition.
A martingale $u$ is a \emph{simple martingale} if there exists some $m \in \Z_+^2$ such that $u_n = \E_m u_n$ for all indices $n \in \Z_+^2$.
With this, we may formulate the following version of Gundy's theorem.
\begin{theorem} \label{th:gundy2d}
Consider a sublinear operator $V$ mapping martingales to measurable functions.
Assume the following two conditions.
\1 The operator $V$ is bounded between $L^2$ and $L^2$.
\2 If $u$ is a simple martingale for which $u_{0, 0} = 0$ and
\[
\Delta_n u = \1_{e_n} \Delta_n u, ~~\t{where} e_n \in \c{F}_m \t{for some} m \leq n, m \not= n
,
\]
then $\cbr{\abs{V u} > 0} \subseteq \bigcup\limits_{n \in \Z_+^2 \setminus \cbr{0}} e_n$.
\0
Then $V$ is bounded between $H^p$ and $L^p$ for any $0 < p \leq 1$.
\end{theorem}
\begin{proof}
Fix $0 < p \leq 1$.
We will show that $V$ is $H^p$ quasi local, then the claim will follow from Theorem~\ref{th:atomic_bounded}.

Take some $H^p$ rectangle atom $a$ supported on a dyadic rectangle $R \subseteq [0,1)^2$.
We need to check that for all $r \in \N$ and for some $\delta$ not depending on $a$ and $r$ we have:
\[ \label{eqn:gundy_to_verify}
\int\limits_{[0,1)^2 \setminus R^r} \abs{V a}^p (x_1, x_2) dx_1 dx_2 \leq C_p 2^{-\delta r}.
\]

We claim that condition \eqref{eqn:gundy_to_verify} can be checked only for atoms that are simple martingales.
Indeed, assume that it is true indeed for all rectangle atoms that are simple martingales, and let us prove that it is true for arbitrary rectangle atom~$a$.
It is easy to check that the simple martingale $a_n = \E_n a$ remains a rectangle atom.
From $\lim_{\min(n_1, n_2) \to \infty} \norm{a_n - a}_{L^2} = 0$ it follows that $\lim_{\min(n_1, n_2) \to \infty} \norm{V a_n - V a}_{L^2} = 0$, because $V$ is bounded between $L^2$ and $L^2$.
Hence $\lim_{\min(n_1, n_2) \to \infty} \norm{V a_n - V a}_{L^p} = 0$.
This justifies the passage to the limit in inequality \eqref{eqn:gundy_to_verify}, which proves inequality for the initial rectangle atom $a$.
Thus hereinafter in this proof we assume all rectangle atoms to be simple martingales.

Now we find an element $N \in \Z_+^2$ such that $R \in \c{F}_N$ and for any $n$ such that $R \in \c{F}_n$, we have $N \leq n$.
Since $\supp a \subseteq R$ and thanks to item 3 in the definition of a rectangle atom, we have
\[
\Delta_{n} a = \1_R \Delta_{n} a \qquad \t{for} n \geq N,\  n \not= N,
\]
\[
\Delta_{n} a = \1_{\emptyset} \Delta_{n} a \qquad \t{otherwise.}
\]
Moreover, $a_{0, 0} = 0$ due to item 3 in the definition of rectangle atom.
Using condition 2 of this theorem, we have $\cbr{\del{V a} > 0} \subseteq R$, hence
\[
\int\limits_{[0,1)^2 \setminus R^r} \abs{V a}^p (x_1, x_2) dx_1 dx_2
\leq
\int\limits_{[0,1)^2 \setminus R} \abs{V a}^p (x_1, x_2) dx_1 dx_2
=
0
.
\]
The right-hand side is trivially bounded by $C_p 2^{-\delta r}$ for any $\delta > 0$, which proves the claim.
\end{proof}

\begin{corollary}
If the conditions of Theorem {\normalfont \ref{th:gundy2d}} are fulfilled, then $V$ is bounded between $L^s$ and $L^s$ for $1 < s \leq 2$.
\end{corollary}
\begin{proof}
Interpolation between the boundedness of $V: H^p \to L^p$ for some $p \leq 1$ and $V: L^2 \to L^2$ by means of Theorem \ref{th:interpolation} gives the result.
\end{proof}

\section{Auxiliary operator $G$}

In this section we introduce the auxiliary operator $G$, the two-parameter counterpart of the auxiliary operator $G$ introduced by Osipov in \cite{osipov2017}, and prove its boundedness using the results of the previous section.
It is this particular operator that will appear in the course of the proof of the main theorem.

Consider a family of multi-indices $\c{A} \subseteq \N \x \Z_+^2$.
Its elements are pairs $(j, k)$, where $j \in \N, k \in \Z_+^2$.
Let $\delta_k$ be as in equation \eqref{eqn:delta_k} and consider a family $\cbr{a_{j, k}}_{(j, k) \in \c{A}} \subseteq \Z_+^2$ such that $\cbr{a_{j, k} \dot{+} \delta_k}_{(j, k) \in \c{A}}$ consists of pairwise nonintersecting subsets of $\Z_+^2$.
We define the operator $G$ induced by the family $\cbr{a_{j, k}}_{(j, k) \in \c{A}}$ and prove its boundedness in the following lemma.

\begin{lemma} \label{th:G_is_bounded}
The operator $G$ maps any vector-valued function $h = \cbr{h_{j, k}}_{(j, k) \in \N \x \Z_+^2}$ from $L^p(l^2_{\N \x \Z_+^2})$, $1 < p \leq 2$, to a measurable function by the following law
\[
(G h)(x_1, x_2) := \sum\limits_{(j, k) \in \c{A}} w_{a_{j, k}}(x_1, x_2) \del{\Delta_{k} h_{j, k}}(x_1, x_2).
\]
This operator is bounded between $L^p(l^2_{\N \x \Z_+^2})$ and $L^p$, that is
\[
\norm{G h}_{L^p} \leq C_p \norm{h}_{L^p(l^2_{\N \x \Z_+^2})},
\]
where the constant $C_p$ depends only on $p$.

\end{lemma}
\begin{proof}
Since for $1 < p \leq 2$ there exists a one-to-one mapping between the elements of $L^p$ and the martingales from $L^p$, the operator $G$ may be viewed as an operator mapping $l^2(\N \x \Z_+^2)$-valued martingales (rather than $l^2(\N \x \Z_+^2)$-valued functions) to measurable functions.
We will prove that $G$ fulfills the conditions of Theorem \ref{th:gundy2d}, or rather, strictly speaking, generalization to the case of $l^2(\N \x \Z_+^2)$-valued martingales.
Here we rely on the fact that everything can be transferred to the $l^2$-valued case, as mentioned at the beginning of Section~\ref{sec:preliminaries}.

$G$ is linear, and thus of course it is sublinear.
Further, Plancherel's identity and the fact that $\cbr{a_{j, k} \dot{+} \delta_k}_{(j, k) \in \c{A}}$ is a family of pairwise nonintersecting sets give the boundedness of $G$ between $L^2$ and $L^2$.

Finally, we claim that for a simple $l^2(\N \x \Z_+^2)$-valued martingale for which $u_{0, 0} = 0$ and $\Delta_n u = \1_{e_n} \Delta_n u$, where $e_n \in \c{F}_m$, $m \leq n, m \not= n$, it we have $\cbr{\abs{G u} > 0} \subseteq \bigcup_{n \in \Z_+^2 \setminus \cbr{0}} e_n$.
Indeed, write
\<
\cbr{\abs{G u} > 0}
\subseteq
\bigcup\limits_{(j, k) \in \c{A}}
    \cbr{
        \abs{
            w_{a_{j, k}}
            \Delta_{k} u_{j, k}
        }
        >
        0
    }
=
\bigcup\limits_{(j, k) \in \c{A}}
    \cbr{
        \abs{
            \Delta_{k} u_{j, k}
        }
        >
        0
    }
\\
=
\bigcup\limits_{(j, k) \in \c{A}}
    \cbr{
        \abs{
            \1_{e_k}
            \Delta_{k} u_{j, k}
        }
        >
        0
    }
\subseteq
\bigcup\limits_{(j, k) \in \c{A}}
    \cbr{
        \abs{
            \1_{e_k}
        }
        >
        0
    }
\subseteq
\bigcup\limits_{k \in \Z_+^2 \setminus \cbr{0}} e_k
.
\>
This proves the claim.
\end{proof}

\section{Proof of the main theorem}

Finally, here we use the theory established in the preceding sections to prove the main theorem.
We begin with recalling its formulation.
\walshrdftwodim*

\begin{proof}
As in \cite{osipov2017}, we partition the rectangles $I_k$ into fragments that behave well under shifts induced by the operation $\dot{+}$.
This, together with Lemma \ref{th:G_is_bounded} and the classical assertion about the boundedness of the Littlewood--Paley square function, will allow us to prove the claim.
Let
\[
I_k= I_k^1 \x I_k^2 = [a_k^{(1)}, b_k^{(1)}-1] \x [a_k^{(2)}, b_k^{(2)}-1]
.
\]

We build the partition of $I_k$ by forming the direct product of partitions of intervals $I_k^1$ and $I_k^2$, while partitioning these individual intervals in the same way as it was done by Osipov in \cite{osipov2017}.

Let us recall that in \cite{osipov2017} an interval $I = [a, b-1] \subseteq \Z_+$ was partitioned into
\[
I
=
\cbr{a}
\u
\del[3]{\bigcup\limits_{j=1}^{r} J_j}
\u
\del[3]{\bigcup\limits_{j=1}^{s} \tilde{J_j}}
=
\del[3]{\bigcup\limits_{j=0}^{r} J_j}
\u
\del[3]{\bigcup\limits_{j=1}^{s} \tilde{J_j}}
,
\]
where $r, s \in \Z_+$ are some numbers, $J_0 = \cbr{a}$, and $J_j$, $\tilde{J_j}$ are pairwise nonintersecting sets.
Moreover, for $j > 0$ we have
$
\abs{J_j} = 2^{\kappa_j}
,
\abs[0]{\tilde{J_j}} = 2^{\gamma_j},
$
where $\kappa_j$ is a strictly increasing and $\gamma_j$ is a strictly decreasing $\Z_+$-valued sequence.

The most important property of the intervals $J_j$ and $\tilde{J_j}$ is that they can be shifted to become the dyadic intervals
\[
a \dot{+} J_0 = \cbr{0}
, \quad
a \dot{+} J_j = [2^{\kappa_j}, 2^{\kappa_j+1}-1]
, \quad
b \dot{+} \tilde{J_j} = [2^{\gamma_j}, 2^{\gamma_j+1}-1]
,
\]
hence the following holds
\<
\label{eqn:1d_decomposition_prop_1}
\Delta_{\kappa_j + 1} w_a f &= w_a f, ~~~\t{if}~ \spec f \subseteq J_j,
\\
\label{eqn:1d_decomposition_prop_2}
\Delta_{\gamma_j + 1} w_b f &= w_b f, ~~~\t{if}~ \spec f \subseteq \tilde{J_j}.
\>

To partition a rectangle $I = I^1 \times I^2 = [a^{(1)}, b^{(1)}-1] \x [a^{(2)}, b^{(2)}-1] \subseteq \Z_+^2$, we partition each interval $I^i$ as above and consider all direct products, yielding
\[
I
=
\del[3]{
    \bigcup\limits_{j}
    A_j
}
\cup
\del[3]{
    \bigcup\limits_{j}
    B_j
}
\cup
\del[3]{
    \bigcup\limits_{j}
    C_j
}
\cup
\del[3]{
    \bigcup\limits_{j}
    D_j
}
,
\]
where
\<
A_j
&=
J_{j^{\del{1}}}^{\del{1}}
\x
J_{j^{\del{2}}}^{\del{2}}
,
\qquad
B_j
=
\tilde{J}_{j^{\del{1}}}^{\del{1}}
\x
\tilde{J}_{j^{\del{2}}}^{\del{2}}
,
\\
C_j
&=
\tilde{J}_{j^{\del{1}}}^{\del{1}}
\x
J_{j^{\del{2}}}^{\del{2}}
,
\qquad
D_j
=
J_{j^{\del{1}}}^{\del{1}}
\x
\tilde{J}_{j^{\del{2}}}^{\del{2}}
,
\>
where a superscript indicates whether the object belongs to the partition of~$I^1$~or~$I^2$.

This partition of $I$ possesses properties similar to those in \eqref{eqn:1d_decomposition_prop_1}, \eqref{eqn:1d_decomposition_prop_2}.
Define $a, b, c, d$ to be the vertices of the rectangle $I$, that is
\[
a := \del[0]{a^{\del{1}}, a^{\del{2}}}
,~~
b := \del[0]{b^{\del{1}}, b^{\del{2}}}
,~~
c := \del[0]{b^{\del{1}}, a^{\del{2}}}
,~~
d := \del[0]{a^{\del{1}}, b^{\del{2}}}
,
\]
then
\<
\label{eqn:spec_prop_1}
\Delta_{\kappa_{j_1}^{\del{1}} + 1, \kappa_{j_2}^{\del{2}} + 1}
w_{a}
f
&=
w_{a}
f
, ~~~\t{if}~ \spec{f} \subseteq A_j,
\\
\label{eqn:spec_prop_2}
\Delta_{\gamma_{j_1}^{\del{1}} + 1, \gamma_{j_2}^{\del{2}} + 1}
w_{b}
f
&=
w_{b}
f
, ~~~\t{if}~ \spec{f} \subseteq B_j,
\\
\label{eqn:spec_prop_3}
\Delta_{\gamma_{j_1}^{\del{1}} + 1, \kappa_{j_2}^{\del{2}} + 1}
w_{c}
f
&=
w_{c}
f
, ~~~\t{if}~ \spec{f} \subseteq C_j,
\\
\label{eqn:spec_prop_4}
\Delta_{\kappa_{j_1}^{\del{1}} + 1, \gamma_{j_2}^{\del{2}} + 1}
w_{d}
f
&=
w_{d}
f
, ~~~\t{if}~ \spec{f} \subseteq D_j.
\>
This behavior under shifts will be of utter importance in what follows.

Let us similarly partition each $I_k$, adding yet another index $k$ to all objects that arise from this construction.
Then $f_k$ can be represented as the sum
\[
f_k
=
\sum\limits_{j} f_{k, j}^A
+
\sum\limits_{j} f_{k, j}^B
+
\sum\limits_{j} f_{k, j}^C
+
\sum\limits_{j} f_{k, j}^D
,
\]
where $\spec f_{k, j}^A \subseteq A_{k, j}$, $\spec f_{k, j}^B \subseteq B_{k, j}$, $\spec f_{k, j}^C \subseteq C_{k, j}$, $\spec f_{k, j}^D \subseteq D_{k, j}$.

Define
\<
g_{k, j}^A
=
w_{a_k}
f_{k, j}^A
,~~~
g_{k, j}^B
=
w_{b_k}
f_{k, j}^B
,~~~
g_{k, j}^C
=
w_{c_k}
f_{k, j}^C
,~~~
g_{k, j}^D
=
w_{d_k}
f_{k, j}^D
.
\>
Then $\sum\limits_{k} f_k$ can be represented as follows:
\<
\sum\limits_{k}
\Bigg(
w_{a_k}
\sum\limits_{j}
g_{k, j}^A
+
w_{b_k}
\sum\limits_{j}
g_{k, j}^B
+
w_{c_k}
\sum\limits_{j}
g_{k, j}^C
+
w_{d_k}
\sum\limits_{j}
g_{k, j}^D
\Bigg).
\>
Application of Lemma \ref{th:G_is_bounded} to this expression (justified by none other than properties \eqref{eqn:spec_prop_1}--\eqref{eqn:spec_prop_4}), followed by applying the triangle inequality, gives us\footnote{Note that $k$ and $j$ here correspond, respectively, to $j$ and $k$ in the formulation of Lemma~\ref{th:G_is_bounded}.}
\<
\norm{\sum\limits_{k} f_k}_{L^p}
\lesssim
\Bigg\lVert
    \Bigg(
        \sum\limits_{k}
        \Bigg(
            \sum\limits_{j}
            \abs{
                g_{k, j}^A
            }^2
            +
            \sum\limits_{j}
            \abs{
                g_{k, j}^B
            }^2
            +
            \sum\limits_{j}
            \abs{
                g_{k, j}^C
            }^2
            \\
            +
            \sum\limits_{j}
            \abs{
                g_{k, j}^D
            }^2
        \Bigg)
    \Bigg)^{1/2}
\Bigg\rVert_{L^p}
\\
\label{eqn:after_G_1}
\leq
\Bigg\lVert
    \Bigg(
        \sum\limits_{k}
            \sum\limits_{j}
            \abs{
                g_{k, j}^A
            }^2
    \Bigg)^{1/2}
\Bigg\rVert_{L^p}
+
\Bigg\lVert
    \Bigg(
        \sum\limits_{k}
            \sum\limits_{j}
            \abs{
                g_{k, j}^B
            }^2
    \Bigg)^{1/2}
\Bigg\rVert_{L^p}
~~~
\\
\label{eqn:after_G_2}+
\Bigg\lVert
    \Bigg(
        \sum\limits_{k}
            \sum\limits_{j}
            \abs{
                g_{k, j}^C
            }^2
    \Bigg)^{1/2}
\Bigg\rVert_{L^p}
+
\Bigg\lVert
    \Bigg(
        \sum\limits_{k}
            \sum\limits_{j}
            \abs{
                g_{k, j}^D
            }^2
    \Bigg)^{1/2}
\Bigg\rVert_{L^p}
.
~~
\>
Hereinafter the symbol $\lesssim$ denotes the inequality up to some implicit multiplicative constant.
Consider now separately, e.g., the third term.
Write
\[
w_{c_k} f_k
=
w_{c_k \dot{+} a_k}
\sum\limits_{j}
g_{k, j}^A
+
w_{c_k \dot{+} b_k}
\sum\limits_{j}
g_{k, j}^B
+
\sum\limits_{j}
g_{k, j}^C
+
w_{c_k \dot{+} d_k}
\sum\limits_{j}
g_{k, j}^D.
\]
We note that $\Delta_{\gamma_{k, j_1}^{\del{1}} + 1, \kappa_{k, j_2}^{\del{2}} + 1} w_{c_k} f_k = g_{k, j}^C$, hence in the decomposition $w_{c_k} f_k = \sum_{n \in \Z_+^2} \Delta_n w_{c_k} f_k$, the functions $g_{k, j}^C$ are among the right-hand side terms.
It follows that
\[
\sum\limits_{j}
\abs{
    g_{k, j}^C
}^2
\leq
\sum\limits_{n \in \Z_+^2}
\abs{
    \Delta_n w_{c_k} f_k
}^2
=
\del[2]{S(w_{c_k} f_k)}^2
,
\]
where $S$ is the Littlewood--Paley square function.
By leveraging its boundedness (cf. the papers \cite{brossard1980,brossard1981,metraux1978}, also the book \cite{weisz2006}, where the scalar-valued version of this statement was proved, from which the vector-valued version follows easily), we have
\<
\Bigg\lVert
    \Bigg(
        \sum\limits_{k}
            \sum\limits_{j}
            \abs{
                g_{k, j}^C
            }^2
    \Bigg)^{1/2}
\Bigg\rVert_{L^p}
\leq
\Bigg\lVert
    \Bigg(
        \sum\limits_{k}
            \sum\limits_{n \in \Z_+^2}
            \abs{
                \Delta_n w_{c_k} f_k
            }^2
    \Bigg)^{1/2}
\Bigg\rVert_{L^p}
\\
\lesssim
\Bigg\lVert
    \Bigg(
        \sum\limits_{k}
        \abs{
            w_{c_k} f_k
        }^2
    \Bigg)^{1/2}
\Bigg\rVert_{L^p}
=
\Bigg\lVert
    \Bigg(
        \sum\limits_{k}
        \abs{
            f_k
        }^2
    \Bigg)^{1/2}
\Bigg\rVert_{L^p}
.
\>
Similarly, we can bound each of the four terms in \eqref{eqn:after_G_1} and \eqref{eqn:after_G_2}.
Collecting these inequalities, we finally obtain
\[
\norm{\sum\limits_{k} f_k}_{L^p}
\lesssim
\Bigg\lVert
    \Bigg(
        \sum\limits_{k}
        \abs{
            f_k
        }^2
    \Bigg)^{1/2}
\Bigg\rVert_{L^p}
,
\]
which proves the claim.
\end{proof}

\begin{remark}
In the light of the papers \cite{osipov2010a} and \cite{osipov2010b}, it is natural to ask whether a similar statement holds for a general multi-parameter Walsh system and a partition of $\Z_+^D$ into arbitrary products of intervals.
The author is going to address this question in the near future.
For now we only mention that there is no direct analog of Theorem \ref{th:atomic_bounded} in the general multi parameter case and a finer statement would be required.
\end{remark}

\printbibliography

\end{document}